\documentclass[12pt]{article}
\usepackage{amsmath,amsfonts,latexsym,amssymb,amsthm}
\usepackage{graphicx}
\newtheorem{thm}{Theorem}
\newtheorem{cor}{Corollary}
\newcommand{\DOIFPDF}[2]{\ifx\pdfoutput\undefined #2\else#1\fi}

\newcommand{\EM}{\ensuremath}

\newcommand{\PDFABLE}[2]{%  
\newif\ifpdf%
\ifx\pdfoutput\undefined\pdffalse%
\else\pdftrue\pdfoutput=1\pdfcompresslevel=9\fi%
\ifpdf%
 \usepackage#1
 \usepackage[pdftex,%
             a4paper,%
             colorlinks,%
             citecolor=blue,%
             pagebackref,%
             plainpages=false]{hyperref}%
\else%
 \usepackage#2
 \usepackage{url}
\fi}
\makeatletter
\newcommand{\@THMSTYLES}{%
  \newtheoremstyle{bodyrm}% name % cf. thmtest.tex shipped with AMSLaTeX
  {4pt}%      Space above
  {4pt}%      Space below
  {}%         Body font
  {}%         Indent amount (empty = no indent,
  {\bfseries\sffamily}% Thm head font
  {.}%        Punctuation after thm head
  { }%        Space after thm head: " " = normal interword space;
  {}%         Thm head spec (can be left empty, meaning `normal')
  \newtheoremstyle{bodyit}% name % cf. thmtest.tex shipped with AMSLaTeX
  {4pt}%      Space above
  {4pt}%      Space below
  {\itshape}% Body font
  {}%         Indent amount (empty = no indent,
  {\bfseries\sffamily}% Thm head font
  {.}%        Punctuation after thm head
  { }%        Space after thm head: " " = normal interword space;
  {}%         Thm head spec (can be left empty, meaning `normal')
}
\newcommand{\THMEN}{%
  \@THMSTYLES
  \theoremstyle{bodyit}
  \newtheorem{thm}{Theorem}[section]%
  \newtheorem{cor}[thm]{Corollary}%
  \newtheorem{prop}[thm]{Proposition}%
  \newtheorem{lem}[thm]{Lemma}%
  \theoremstyle{bodyrm}%
  \newtheorem{defi}[thm]{Definition}%
  \newtheorem{xpl}[thm]{Example}%
  \newtheorem{exo}[thm]{Exercise}%
  \newtheorem{hyp}[thm]{Hypothesis}%
  \newtheorem{eur}[thm]{Heuristics}%
  \newtheorem{pro}[thm]{Problem}%
  \newtheorem{rem}[thm]{Remark}%
  \newtheorem{prp}[thm]{Property}%
}
\newcommand{\THMFR}{%
  \@THMSTYLES
  \theoremstyle{bodyit}
  \newtheorem{thm}{Théorème}[section]%
  \newtheorem{cor}[thm]{Corollaire}%
  \theoremstyle{bodyrm}%
}
\makeatother

\makeatletter
\newcommand{\SMALLSECS}{%
 \renewcommand{\section}{\@startsection%
  {section}%                           % name
  {1}%                                 % level
  {0em}%                               % indent
  {\baselineskip}%                     % beforeskip
  {0.5\baselineskip}%                  % afterskip
  {\normalfont\large\bfseries}}%       % style
 \renewcommand{\subsection}{\@startsection%
  {subsection}%                        % name
  {2}%                                 % level
  {0em}%                               % indent
  {\baselineskip}%                     % beforeskip
  {0.25\baselineskip}%                 % afterskip
  {\normalfont\bfseries}}%             % style
}
\makeatother

\makeatletter
\providecommand{\timenow}{\@tempcnta\time
\@tempcntb\@tempcnta
\divide\@tempcntb60
\ifnum10>\@tempcntb0\fi\number\@tempcntb
\multiply\@tempcntb60
\advance\@tempcnta-\@tempcntb
:\ifnum10>\@tempcnta0\fi\number\@tempcnta}
\makeatother

\makeatletter
\newcommand{\versiondetravail}{%
 \renewcommand{\@evenfoot}{%
 \hfil{\tiny\texttt{%
   Version préliminaire, compilée le \today{} à \timenow.}\hfill}}%
 \renewcommand{\@oddfoot}{\@evenfoot}%
}
\makeatother

\newcommand{\cJ}{\EM{\mathcal{J}}}

\newcommand{\bP}{\EM{\mathbf{P}}}

\newcommand{\bR}{\EM{\mathbf{R}}}

\newcommand{\bX}{\EM{\mathbf{X}}}
\newcommand{\bY}{\EM{\mathbf{Y}}}

\newcommand{\la}{\lambda}

\newcommand{\f}{\frac}

\newcommand{\Det}[1]{\mathrm{Det}\,}

\renewcommand{\leq}{\leqslant}
\renewcommand{\geq}{\geqslant}

\begin{document}
%
%
%\centerline{EXAMPLE}
\vskip 3mm
$l^1$ PENALTY FOR ILL-POSED INVERSE PROBLEMS
\noindent

\vskip 3mm

\vskip 5mm
\noindent Jean-Michel Loubes

\noindent Institut de
  Math\'ematiques, Equipe de statistique et de probabilit\'es UMR C5219 CNRS

\noindent Universit\'e Toulouse 3

\noindent  31062, Toulouse, \textsc{Cedex} 9, France

\noindent Loubes@cict.fr

\vskip 3mm

\noindent ABSTRACT

We tackle the problem of recovering an unknown signal observed in an ill-posed inverse problem framework. More precisely, we study a procedure commonly used in numerical analysis or image deblurring: minimizing an empirical loss function balanced by an $l^1$ penalty, acting as a sparsity constraint. We prove that, by choosing a proper loss function, this estimation technique enables to build an adaptive estimator, in the sense that it converges at the optimal rate of convergence without prior knowledge of the regularity of the true solution.
\vskip 4mm

\noindent    INTRODUCTION AND NOTATION

In this article we are interested in recovering an unobservable
signal $x_0$ based on observations  \begin{equation} \label{eq:model}
y(t_i)=F(x_0)(t_i)+\varepsilon_i, 
\end{equation}
where $F:X\to Y$ is a
linear functional, with $X,\: Y$ Hilbert spaces and
$t_i,i=1,\ldots,n$ is a fixed observation scheme. $x_0:  \bR \to
\bR$ is the unknown function to be recovered from the data
$y(t_i),\: i=1,\dots,n$. The regularity condition over the
unknown parameter of interest is expressed through the assumption
$x_0 \in X$ and will be made precise later in Section 3. We assume that the observations $y(t_i) \in \bR$
and that the observation noise $\varepsilon_i $ are i.i.d.
realizations of a certain random variable $\varepsilon$.
Throughout   the paper, we shall denote ${\bf
y}=(y(t_i))_{i=1}^n$. We assume $F$ is Fr{\'e}chet differentiable
and ill-posed in the sense that  our noise corrupted observations
might lead to large deviations when trying to estimate $x_0$.\vskip .1in In
a deterministic framework, the statistical model \eqref{eq:model}
is formulated  as the problem of approximating the  solution of
$$F(x)=y,$$  when $y$ is not known, and is only available through  an approximation $y^\delta$, $$\|y-y^\delta\|\leq \delta.$$
It is important to remark that whereas in this case consistency of the estimators depends on the approximation parameter $\delta$, in (\ref{eq:model}) it depends on the number of observations $n$.\\
\indent The best $L^2$ approximation of $x_0$ is
$x^\dag=F^{\dag}y$, where $F^{\dag}$ is the Moore-Penrose (generalized)
inverse of $F$. We will say the problem is ill-posed if   $F^\dag$
is unbounded. This entails that $F^\dag(y^\delta )$ is not close to $x^\dag$.
 Hence, the inverse operator needs to be, in some sense, regularized.  Regularization methods replace an
ill-posed problem by a family of well-posed problems. Their
solution, called regularized solutions, are used as approximations
of the desired solution of the inverse problem. These methods
always involve some parameter measuring the closeness of the
regularized and the original (unregularized) inverse problem.
Rules (and algorithms) for the choice of these regularization
parameters as well as convergence properties of the regularized
solutions are central points in the theory of these methods, since
they allow to find the right balance between stability and
accuracy. \vskip .1in
For this we consider
penalized M-estimators minimizing quantities of the form
\begin{equation} \label{nou} \hat{x}_n= {\rm arg} \min_{x \in {\mathcal{X}}} \left( \gamma_n(y-F(x)(t))+
\alpha_n {\rm pen}(x,{\mathcal{X}}) \right),
\end{equation}
where ${\mathcal {X}}$ is a specific set, $\gamma_n(.)$ is an empirical
loss-function, ${\rm pen}(.,.)$ is a penalty over $x$ in
${\mathcal{X}}$, and $\alpha_n\in \Theta$ is a decreasing sequence
all of which   will be defined precisely later. The idea of
penalized M-estimators is to find
 an estimator close enough to the data, close in the sense defined by $\gamma_n$ and with a regularity
 property induced by the choice of the penalty ${\rm pen}$. The smoothing sequence $\alpha_n$ balances the two terms.
  The greater $\alpha_n$, the smoother  the estimator will be,
  while the smaller $\alpha_n$ the closer the estimator will be to the
  data, maybe leading to a too rough estimate. Moreover the penalty should be chosen without prior knowledge of the regularity of the function to be estimated, $x_0$, in order to give rise to adaptive estimator. Adaptivity here has to be understood in the sense that the estimator converges at the optimal rate  without knowing a priori its regularity. It is a theoretical alternative to cross validation methods developed for instance in Dey et al. (1996). We point out that in numerical analysis, adaptivity is not a key issue as in statistics. Indeed the smoothing sequence is often selected using such posterior techniques, see Tautenhahn and Jin (2003) or Kaltenbacher (2000). In a deterministic setting, the rates are not altered. This not the case when observing the data in a white noise framework, which justifies the need for fully adaptive methods.\vskip .1in
Different choices of penalty have been investigated in the literature.  The traditional choice of a quadratic penalty defines the Tikhonov regularized estimator whose behaviour is well studied but which does not lead to adaptive estimation. For general references about this estimator, we refer to Tikhonov et al. (1998), Bissantz et al. (2004), or Engl (2000). A penalty on the number of non zero coefficients leads to hard-thresholded or projection estimators, whose asymptotic behaviour is studied in Kaltenbacher (2000), Mair and Ruymgaart (1996) or Engl et al. (1996). Adaptive estimators can yet be built with both methodologies, but by considering model selection techniques, as done in Loubes and Lude\~na (2005). But adaptivity means, in that case, that the estimator behaves as well as the best estimate obtained over a fixed class of estimators, i.e a collection of models. This property is generally expressed through an oracle inequality. Hence, there might be a bias if the true solution is not well approximated by the sieves. \\
\indent In this article, we tackle the problem of the asymptotic behaviour of the estimator obtained with a $l^1$ penalty. Indeed, over the last decade, $l^1$ penalty has been more and more used in a large variety of fields. Indeed such a penalty selects sparse signals in a smoother way than hard-thresholding penalty and can also be easily implemented. Contrary to differentiable penalties for which adaptivity implies selecting the smoothing sequence among a set of possible choices, there is an optimal choice of the trade-off parameter when using a soft-thresholding penalty. And this optimal choice does not depend on the unknown regularity of the parameter of interest, which enables adaptive estimation. Hence $l^1$ norm penalty is used in estimation with soft-thresholding estimators in Loubes and van de Geer (2002) or Loubes (2007), and in inverse problems in Daubechies et al. (2004) or Cohen et al. (2003). We point out that in Cohen et al. (2003), the estimator is the soft thresholded version of the estimation which is used, and the properties of the estimator are studied using a sequential version of model \eqref{eq:model} together with a projection method.\vskip .1in
 In this article, we construct a penalized estimator with a $l^1$ penalty, with an appropriate loss function depending on the operator. We show that such an estimator converges and is adaptive over a class of Besov spaces for a certain class of ill-posed problems. When trying to minimize a standard empirical quadratic loss function together with a softhresholding penalty, as often done in numerical analysis literature, the results are different. The estimator may be inconsistent and its rate of convergence is rather slow. We also provide the rate of convergence in this case and stress the advantages of choosing a loss function depending on the operator to build a more efficient estimator.\vskip .1in
The article falls into the following parts. Section 1 presents the model and the overall assumptions. The estimation procedure and its general efficiency are described in Section 2. Rates of convergence and adaptivity under smoothness assumptions are also discussed in Section 3. The section 4 is devoted to the analysis of the performances of the usual penalized least squares estimator for inverse problems. Simulations are conducted in Section 5.

\vskip 3mm

\noindent 1. INVERSE PROBLEM MODEL 

Consider the following inverse model:
\begin{equation}
  \label{eq:model2}
  y_i=y(t_i)=F(x_0)(t_i)+ \epsilon_i, \: i=1,\dots,n.
\end{equation}
 where $F: \bX \rightarrow \bY$ is a known linear operator, whose adjoint will be denoted $F^*$. 
 Set $\delta$ the Dirac function and define $Q_n$ the empirical measure of the
covariables as $Q_n = \frac{1}{n} \sum_{i=1}^n \delta_{t_i} .$ Throughout all the paper, the estimation errors will be given with respect to the $L_2 ( Q_n )$-norm  defined, for all  functions $y \in \bY$, by
$$\| y \|^2_{n}=  \int y^2 d Q_n  = \frac{1}{n} \sum_{i=1}^n y^2(t_i).$$
The corresponding empirical scalar product is given by $<y,\epsilon >_n = \frac{1}{n} \sum_{i=1}^n \epsilon_i y(t_i).$ \vskip .1in
As often $F$ is not of full rank, so the singular value
decomposition (SVD) of the operator $S$ is then a useful tool.\\
Let $(\lambda_j; \varphi_j, \psi_j)_{j \geq 1}$ be a singular system
for the linear operator $F$, that is, $F \psi_j = \lambda_j
\varphi_j$ and $F^* \varphi_j = \lambda_j \psi_j$; where
$\{\lambda_j^2\}_{j \geq 1}$ are the non zero eigenvalues of the
selfadjoint operator $F^* F$,
considered in decreasing order. Furthermore, $\{\psi_j\}_{j =1,\dots,n
}$ and $\{\varphi_j\}_{j =1,\dots,n}$ are a corresponding complete
orthonormal system with respect to $\|.\|_n$ of eigenvectors of $F^*F$ and $F F^*$,
respectively. For general linear operators with an SVD
decomposition, we can write for all $(x,y) \in \bX \times \bY$

\begin{equation} F x = \sum_{j=1}^n \lambda_{j} \langle x , \psi_j\rangle
\varphi_j  
\end{equation}

\begin{equation} F^* y = \sum_{j=1}^n \lambda_{j} \langle y ,
\varphi_j\rangle \psi_j.
\end{equation}
\vskip .1in

For $y$  in the domain of  $F^{\dag}$,
$\mathcal{D}(F^{\dag})$, the best-approximate $L^2$ solution
has the expression
$$F^{\dag}y = \sum_{j=1}^n \frac{\langle y ,\varphi_j\rangle
}{\lambda_j} \psi_j = \sum_{j=1}^n \frac{\langle F^*y
,\psi_j\rangle}{\lambda_j^2} \psi_j.$$

Note that for large $j$, the term $1/\lambda_j$ grows to infinity.
Thus, the \emph{high frequency errors} are strongly amplified. This amplification measures the difficulty of the inverse problem, the faster the decay of the eigenvalues, the more difficult is the inverse problem. In this paper we will tackle the problem of polynomial decay of eigenvalues, which can be described by the following assumption  
\begin{description}
\item[Index of ill-posedness] 
Assume that there exists an index $t>0$, called  the index of ill-posedness of the
operator $F$, following notations in Engl et al. (1996), such that
 $$\lambda_j = \mathcal{O}(j^{-t}).$$ 
\end{description}
This difference with standard regression model for which $t=0$ yields other optimal rates  which are usual in statistics. In Section 3, we will compare the rates obtained by our estimator to these optimal rates of convergence. 
\vskip .1in
A penalized M-estimator is defined using a distance $d$, between the observations $y=(y(t_1),\dots,y(t_n)) \in Y$ and a function $x \in X$, and a penalty. Hence we shall study an estimator of the following type 
\begin{equation}
  \label{eq:1}
  \hat{x}_n= {\rm arg}\min_{x \in X} \left[ {\frac 1 n} \sum_{i=1}^n d(y(t_i),x(t_i)) + I_\mu( x) \right],
\end{equation}
where $\mu$ is a smoothing sequence and $I_\mu(.)$ is taken to be the soft-thresholding penalty. Indeed for $\mu=(\mu_j),\: j=1,\dots,n$ and $\forall x=\sum_{j=1}^n x_j \psi_j \in X,$ set $ \: I_\mu(x):=\sum_{j=1}^n |\mu_j x_j|$  the $l^1$ weighted norm of the function $x$. In the direct case where $F={\rm Id}$, and for a quadratic loss function, i.e $${\frac 1 n} \sum_{i=1}^n d(y(t_i),x(t_i)) =  {\frac 1 n} \sum_{i=1}^n |y(t_i)-x(t_i)|^2,$$ the estimator can be computed explicitly and is called the soft-thresholded estimator, as pointed out in Loubes and van de Geer (2002). \vskip .in In the inverse problem literature, such an estimator is often used in image recognition or in geophysics. Indeed the $l^1$ norm is well fitted to handle such signals. In such cases, it is not possible to solve the corresponding minimization issue and the asymptotic properties of the estimator are not known. Hence, in the following section, we propose an adequate choice of loss function which enables to build a sparse estimator converging at an optimal rate of convergence. Moreover this estimator is adaptive, thanks to the sparsity property of the $l^1$ norm. We point out that adaptivity means that the regularity of the function $x_0$ is unknown but that the estimator still achieves the optimal rate of convergence. Nevertheless, the operator $F$ is assumed to be known, as well as the degree of ill-posedness. This assumption is common in the statistical literature on inverse problems.

\vskip 3mm

\noindent 2. ESTIMATION USING SOFT THRESHOLDING PENALTY

In this section, we investigate the classical inverse regression model \eqref{eq:model} with  independent errors $\epsilon_1,\dots,\epsilon_n$ with zero expectation and finite variance $\sigma^2$. Assume moreover that the following condition over the observations errors holds
\begin{description}
\item[Error bound ($\Delta$):] 
Suppose that for some constant $K < \infty$, 
the errors
$\epsilon_1 , \ldots , \epsilon_n$ satisfy
$$\max_{i=1 , \ldots , n } {\bf E} \exp [ \epsilon_i^2 / K^2 ] \leq K .$$
\end{description}
 For a choice of penalty $I_\mu(x)=\sum_{j=1}^n \mu_j |x_j|$, let $\hat{x}_n$ be the $l^1$ penalized estimator defined as
\begin{align}
  \label{eq:2}
  \hat{x}_n & ={\rm arg}\min_{x=\sum_{j=1}^n x_j \psi_j \in X} \left[ \sum_{j=1}^n \left| <y-F(x), \frac{\varphi_j}{\lambda_j}>_n \right|^2 +  I_\mu( x) \right]\\
 & = \sum_{j=1}^n \hat{x}_{j,n} \psi_j \nonumber
\end{align}
 with $\mu=(\mu_j,\: j=1,\dots,n)$ a sequence of real numbers.\\
\indent This penalized estimator mimics the soft-thresholded estimator, as pointed out in Loubes and van de Geer (2002). Contrary to the direct case where the thresholding level can be chosen equal to a constant $\mu=\mu_n$, in this case we consider a threshold that changes at each reconstruction level $j$, $\mu=(\mu_j)$ and which depends on the nature of the inverse problem. Indeed, selecting the coefficients without considering the effect of the inverse problem is too rough and the usual choice in nonparametric estimation $\mu=2c \sqrt{ \frac  {\log n} n}$ may lead to inconsistent estimator. Hence, in this this work, define the smoothing sequence as $\forall j=1,\dots,n\:$ $\mu_j:=2\frac{c}{\la_j} \sqrt{ \frac {\log n} n},$ for $c$ a given constant.\\
\indent Write $x_0=\sum_{j=1}^n x_{j,0} \psi_j$ and consider  $x_{*}=\sum_{j=1}^n x_{j,*} \psi_j$, the hard-threshold version of the true function $x_0$, defined as
$$ x_{j,*}= \begin{cases} x_{j,0}, \:  & {\rm if}\: |x_{j,0}| > \mu_j \\
0, \: &  {\rm if} \: |x_{j,0}| \leq \mu_j 
\end{cases}, \: j=1,\dots,n.
$$ We will now establish an upper bound for $\| \hat{x}_n-x_0\|_n$ which depends on the performance of the oracle $\|x_*-x_0\|_n$. This will enable us to get rates of convergence for ill-posed inverse problems and to prove adaptivity of the estimation procedure.\\
\indent In the theorem we write
$$ V_j=\frac{1}{n} \sum_{i=1}^n \varphi_j(t_i) \epsilon_i,\: j=1,\dots,n.$$
\begin{thm} 
\label{first}
  Let $B_n$ be the set $$B_n= \{ \max_{j=1,\dots,n} |V_j| \leq c \sqrt{\frac{\log n}{n}} \}.$$ Consider the set of indexes $\mathcal{J}_n= {\rm Card} \{j, \: |x_{j,0}|> \mu_j \}.$ Then on $B_n$ we have
  \begin{equation}
    \label{tradeoff}
    \| \hat{x}_n-x_0\|_n^2 \leq \| x_{*}-x_0\|_n^2 + 4 c \sqrt{\frac{\log n}{n}} \sum_{j \in \cJ_n} \frac{|\hat{x}_j-x_{j,0}|}{\lambda_j},
  \end{equation}
\end{thm}

\begin{proof}
  First note that the empirical contrast can be rewritten in a different way for all $x\in X$
\begin{align*}
  \sum_{j=1}^n \left| \frac{1}{\lambda_j} <y-F(x),\varphi_j>_n \right|^2 & = \sum_{j=1}^n \left| \frac{1}{\lambda_j} [ <y-F(x_0),\varphi_j>_n+<F(x_0)-F(x),\varphi_j>_n] \right|^2 \\
 = \sum_{j=1}^n \left| \frac{1}{\lambda_j} <\epsilon,\varphi_j>_n \right|^2 & + \sum_{j=1}^n |x_{j,0}-x_j|^2 + 2 \sum_{j=1}^n \frac{|x_{j,0}-x_j|}{\lambda_j}<\epsilon,\varphi_j>_n.
\end{align*}
 Set $ V_j=<\epsilon,\varphi_j>_n=\frac{1}{n} \sum_{i=1}^n \epsilon_i \varphi_j(t_i).$ 
Using Definition \eqref{eq:2} and previous remark,  it implies that
\begin{align*}
\sum_{j=1}^n \left| \frac{1}{\lambda_j} <y-F(\hat{x}_n),\varphi_j>_n \right|^2 +   I_\mu (\hat{x}_n)  & \leq \sum_{j=1}^n \left| \frac{1}{\lambda_j} <y-F(x_*),\varphi_j>_n \right|^2 +  I_\mu (x_*) \\
\| \hat{x}_n - x_0\|_n^2 +  I_\mu (\hat{x}_n) & \leq \|x_*-x_0\|_n^2+2 \sum_{j=1}^n V_j \frac{|\hat{x}_j-x_{j,*}|}{\lambda_j} + I( \mu x_*) \\
& \leq  \|x_*-x_0\|_n^2+ 2 \left(\max_{j=1,\dots,n}|V_j|\right) \sum_{j=1}^n \frac{|\hat{x}_j-x_{j,*}|}{\lambda_j} + I_\mu( x_*)
\end{align*}
Now use the following decomposition $$ I_\mu(x)= \sum_{j \in \cJ_n} \mu_j|x_j|+ \sum_{j \notin \cJ_n} \mu_j|x_j|,$$ to get that  on the set $B_n$
\begin{align*}
  \| \hat{x}_n - x_0\|_n^2 + I_\mu( \hat{x}_n ) & \leq  \|x_*-x_0\|_n^2+ 2 c \sqrt{\frac{\log n}{n}} \sum_{j=1}^n \frac{|\hat{x}_j-x_{j,*}|}{\lambda_j} + I_\mu (x_*) \\
\| \hat{x}_n - x_0\|_n^2 + \sum_{j \notin \cJ_n} \mu_j |\hat{x}_j|  & \leq  \|x_*-x_0\|_n^2 + 2 c \sqrt{\frac{\log n}{n}} \sum_{j \in \cJ_n} \frac{|\hat{x}_j-x_{j,*}|}{\lambda_j} +  2 c \sqrt{\frac{\log n}{n}} \sum_{j \notin \cJ_n} \frac{|\hat{x}_j|+|x_{j,*}|}{\lambda_j} \\
 + &  \sum_{j \in \cJ_n} \mu_j (|x_{j,*}|-|\hat{x}_j|) +  \sum_{j \notin \cJ_n} \mu_j |x_{j,*}|.
\end{align*}
Point out that 
$$
  x_{j,*}= \begin{cases} x_{j,0}, \: &{\rm if} \:  j \in \cJ_n \\
 0, \: & {\rm if}\:  j \notin \cJ_n 
\end{cases}
.
$$
Then, for a choice $\mu_j=2\frac{c}{\lambda_j}\sqrt{\log n /n}$ and using $||\hat{x}_j|-|x_{j,0}||\leq |\hat{x}_j-x_{j,0}|$, we obtain that:
\begin{align*}
\| \hat{x}_n - x_0\|_n^2 + \sum_{j \notin \cJ_n} \mu_j |\hat{x}_j| & \leq  \|x_*-x_0\|_n^2+4 c \sqrt{\frac{\log n}{n}} \sum_{j \in \cJ_n} \frac{|\hat{x}_j-x_{j,0}|}{\lambda_j}  +  \sum_{j \notin \cJ_n} \mu_j |\hat{x}_j| \\
\| \hat{x}_n - x_0\|_n^2 & \leq  \|x_*-x_0\|_n^2+4 c \sqrt{\frac{\log n}{n}} \sum_{j \in \cJ_n} \frac{|\hat{x}_j-x_{j,0}|}{\lambda_j},
\end{align*}
which proves the result.
\end{proof}

%%%%

\begin{cor} \label{result} Under assumption $(\Delta)$, it follows from e.g.\ { van de Geer} (2000) , Lemma 8.2), that
for a constant $c$ depending on $K$,
\begin{equation} \label{geer} {\bf P} \left ( \max_{j=1 , \ldots , n} | V_j| > c \sqrt {\log n 
\over n } \right ) \leq c \exp [ - { \log n \over c^2 } ] . 
\end{equation}
Thus, we 
obtain for two positive finite constants $c_1$ and $c_2$
$${\bf P} \left( \| \hat{ x}_n - x_0 \|_n^2 >c_1
 \| x_* - x_0 \|_n^2 +  c_2 \frac{\log n}{n} \sum_{j \in \cJ_n}\frac{1}{\lambda_j^2} \right) \leq
c \exp [- { \log n \over c^2 } ] . $$ 
\end{cor}
We point out that previous bound is as sharp as the equivalent one for direct estimation problems, see for instance Loubes and van de Geer (2002). The choice of the smoothing sequence $\mu$ does not depend on the regularity of the unknown function $x_0$. Hence we expect adaptivity under regularity conditions. It still depends on the distribution of the errors, since the constant $c$ in \eqref{geer} may be large. As a consequence, if the errors have heavy tails, the rate of convergence of the $l^1$ penalized estimator may be slow.\vskip .1in
Nevertheless, this estimator works with real observation and can handle a large variety of observation noise since Assumptions $(\Delta)$ are rather weak.
\begin{proof}
  First note that Cauchy-Schwartz inequality entails that, on the set $B_n$, we have
$$\| \hat{x}_n - x_0\|_n^2  \leq 4 c \sqrt{\frac{\log n}{n}} \| \hat{x}_n-x_0 \|_n \sqrt{\sum_{j \in \cJ_n} \frac{1}{\lambda_j^2}} +  \| x_* - x_0 \|_n^2.$$

  Since for all $1>\gamma>0$,  we get  $2xy\leq \frac{1}{\gamma}x^2 + \gamma y^2$, we obtain that
  \begin{equation*}
(1-\gamma)    \|\hat{x}_n-x_0\|_n^2 \leq  \| x_* - x_0 \|_n^2+ \frac{4c^2}{\gamma} \frac{\log n}{n} \sum_{j \in \cJ_n}\frac{1}{\lambda_j^2},
  \end{equation*}
leading to \eqref{result} for proper constant choices.
\end{proof}
\vskip 3mm

\noindent 3.  RATES OF CONVERGENCE UNDER SOURCE CONDITION

In this section, we illustrate the consequences of Corollary \ref{result} for functions $x_0$ belonging to some special smoothness sets. Consider the set of functions defined by two parameters, a smoothness parameter $s$ and a moment parameter $0<p<2$ as 
\begin{equation} \label{Besov2} X_{s,p}=\{x=\sum_{j=1}^n x_j\psi_j, \: \sum_{j=1}^n j^{p\left(s+\frac{1}{2}-\frac{1}{p}\right)}x_j^p \leq 1\}.
\end{equation}
Such sets are balls of Besov bodies associated to the Besov spaces $B_{pp}^s([0,1]).$ These spaces are intrinsically connected to the analysis of curves since the scale of Besov spaces yields the opportunity to describe the regularity of functions, with more accuracy than the classical H\"older scale. General references about Besov spaces are Besov et al. (1978). Consider a wavelet basis $(\psi_{jk})_{j,k}$ with regularity $r$ such that $r\geq s$. Then a Besov norm is equivalent to an appropriate norm in the sequence space, that is the space of the wavelet coefficients. If $x_{jk}$ are the wavelet coefficients of a function $x=\sum_{j,k} x_{jk} \psi_{jk}$, hence the ball with radius 1 of Besov space $B^s_{pp}([0,1])$ can be fully characterized by the Besov semi-norm
\begin{equation}
  \label{Besov1}
  \left( \sum_{j=1}^{+\infty} \left[2^{j(s+\frac{1}{2}-\frac{1}{p})} \left( \sum_{k=0}^{2^j-1} x_{jk}^p\right)^{\frac{1}{p}}\right]^p\right)^{\frac{1}{p}} \leq 1.
\end{equation}
as proved in H\"ardle et al. (1998). In the Besov space interpretation $X_{s,p}$ with $n=\infty$ corresponds (in the sense of norm equivalence) to a Besov ball in the space $B^s_{pp}([0,1])$.\vskip .1in
Here consider the special case where $p$ is such that $\f 1 p=\f 1 2+ \f s {2t+1}$. This choice corresponds to the set for which classical soft thresholded estimators for inverse problems are optimal, see for instance Theorem 3.2 in Cohen et al. (2003). For that choice of parameters, we have $j^{p(s+1/2-1/p)}=\lambda_j^{-2sp/(2t+1)}$. Hence previous set can then be rewritten as $$X_{s,p}= \{ x=\sum_{j=1}^n x_j \psi_j,\: \sum_{j=1}^n \frac{x_j^p}{\lambda_j^{\f {2sp}{2t+1}}} \leq 1 \}.$$
Such regularity condition can be interpretated as a source set condition, used in the literature of deterministic inverse problems, see for instance Engl et al. (1996), Darolles et al. (2003) or Fermin et al. (2005). Such spaces link the decay of the $x_j$'s, the coefficients of the unknown function in the SVD basis with the decay of the $\la_j$'s, the eigenvalues of the operator.
% Set $\alpha= \f {sp} {2t+1}$ and consider the source set $$ X_{\alpha,2}=\cR\left((F^*F)^\alpha \right)= \{ x=\sum_{j=1}^%n x_j \psi_j, \: \sum_j \frac{ x_j^2}{ \lambda_j^{2\alpha}} < +\infty \}.$$
%Hence Assumption \eqref{Besov2} is the extension to the non Hilbert case of the source condition regularity assumption.
\vskip .1in
The following theorem gives the rate of convergence of the penalized $l^1$ estimator for inverse problems.
\begin{thm} \label{thrates}
  Assume that there are $s$ and $0<p<2$ such that $x_0 \in X_{s,p}$, with $\f 1 p=\f 1 2+ \f s {2t+1}$. Then we get the following estimation error
  \begin{equation}
    \label{eq:rate}
    \| \hat{x}_n-x_0 \|^2_n = O_{\bP} \left( \left(\f n  {\log n}\right)^{-\frac{2s}{2s+2t+1}} \right).
  \end{equation}
\end{thm}
\begin{proof}
  Starting from Corollary \ref{result}, we have to bound the two terms $\|x_*-x_0\|_n^2$, which stands for a bias term, and $ \sum_{j \in \cJ_n}\frac{1}{\lambda_j^2}$ which stands for a variance term. For this we will take $\cJ_n=\{j,\:x_{j,0} \geq \frac{1}{\lambda_j}\sqrt{\frac{\log n}{n}} \}$ and make use of the assumption $\sum_{j \geq 1} \lambda_j^{-2sp/(2t+1)}x_{j,0}^p \leq 1.$ \\ First write
\begin{align*}
  \|x_*-x_0\|_n^2 & = \sum_{j \notin \cJ_n} x_j^2 \\
& = \sum_{j \notin \cJ_n} x_j^{2-p}x_j^p j^{\f {2tps}{2t+1}} j^{-{\f {2tps}{2t+1}}} \\
& \leq \left( \frac{\log n}{n}\right)^{1-\f p 2} \sum_{j=1}^n x_j^p j^{\f {2tps}{2t+1}} j^{t(2-p)-{\f {2tps}{2t+1}}} \\
& \leq \left( \frac{\log n}{n}\right)^{1-\f p 2}=\left( \frac{\log n}{n}\right)^{\frac{2s}{2s+2t+1}}. 
\end{align*}
using repeatdly  $\f 2 p =1+ \f {2s} {2t+1}$ and $t(2-p)-{\f {2tps}{2t+1}}=0$.\\
\indent On the other hand we get 
\begin{align*}
  \sum_{j \in \cJ_n} \frac{1}{\lambda_j^2} & = \sum_{j \in \cJ_n} j^{t(2-p)}j^{tp} \\
 & \leq \left( \frac{n}{\log n}\right)^{\f p 2} \sum_{j=1}^n x_j^p j^{\f {2tps}{2t+1}} j^{t(2-p)-\f {2tps}{2t+1}} \\
 & \leq \left( \frac{n}{\log n}\right)^{\f p 2},
\end{align*}
 Previous upper bounds lead to \eqref{eq:rate}, concluding the proof.
\end{proof}
We point out that  we obtain the optimal rate of convergence for ill-posed inverse problems, i.e $\left( n \right)^{-\frac{2s}{2s+2t+1}} $, up to a logarithmic factor. Hence, the  estimator \eqref{eq:2} is adaptive with respect to the parameter $s$ within the range of Besov spaces $B^s_{pp}([0,1])$ with $\f 1 p=\f 1 2+ \f s {2t+1}$ for the empirical quadratic loss. This result is the same as the one obtained in Cavalier et al. (2002), when working in the sequential model. Hence we provide a new estimator which achieves optimal rates of convergence and which can be easily computed, as shown in Section 5. 
\vskip 3mm

\noindent 4. COMMENTS ON PENALIZED LSE FOR INVERSE PROBLEMS

The estimator obtained in \eqref{eq:2} uses a specific loss function adapted to the particular ill-posed problem, namely $d_n(y,x):= \sum_{j=1}^n \left| <y-F(x),\f {\varphi_j}{ \lambda_j}>_n \right|^2.$ This loss function leads to optimal rate of converge but its main drawback is that the knowledge of the operator $F$ and its SVD $(\lambda_j;\varphi_j,\psi_j),\: j\geq 1$ are needed. In numerical analysis for inverse problems, a classical deblurring procedure involves minimizing the usual quadratic empirical loss function together with the $l^1$ penalty, see for instance Daubechies et al. (2004) or Cohen et al. (2003). More precisely an estimator is defined as
\begin{equation}
  \label{eq:3}
  \tilde{x}_n = {\rm arg}\min_{x=\sum_{j=1}^n x_j \psi_j} \left(\| y-F(x) \|_n^2 +  I_\mu (x ) \right).
\end{equation}
This estimation procedure is used and its weak consistency is well-known, even if it is known to provided sometimes an inconsistent estimate. However this estimator can be easily implemented in an iterative procedure, similar to a gradient descent algorithm with a data driven step. Such algorithm is widely used in image deblurring for instance. We refer to Daubechies et al. (2004) for more references.\\
\indent The following theorem  gives conditions to ensure consistency and provides rates of convergence for the empirical loss function.
\begin{thm} \label{pasbien}
Assume that the errors $\epsilon$ satisfy to the condition $(\Delta)$.
Assume that there exists a roughness parameter $0<\rho<2$ such that \begin{equation} \sum_{j=1}^n |x_{j,0}|^\rho \leq 1. \label{reg} 
\end{equation}
Hence the estimator defined in \eqref{eq:3} is consistent as soon as $1-\rho/2-2t >0$ and converges at the following rate of convergence
\begin{equation}
  \label{tilde}
  \| \tilde{x}_n-x_0\|_n^2 = O_{\bf P} \left(\log^{1-\f \rho 2}(n) \left(\frac{1}{n}\right)^{ 1 -\f \rho 2 -2t} \right).
\end{equation}
\end{thm}
\begin{proof}
  Using the definition of the estimator, first point out that on the set $B_n= \{ \max_{j=1,\dots,n} |V_j| \leq c \sqrt{\frac{\log n}{n}} \},$ we have following the guidelines of the proof of Theorem \ref{first}
$$ \|y-F(\tilde{x}_n)\|_n^2 +  I_\mu (\tilde{x}_n ) \leq \|y-F({x}_\star)\|_n^2 + I_\mu ({x}_\star),$$ leading for $x_\star$ the oracle and a choice $$\cJ_n=\{j,\: |x_{j,0}| \geq c \frac{1}{\lambda_j} \sqrt{\frac{\log n}{n}}\},$$
  \begin{align*}
    n^{-2t} \|\tilde{x}_n-x_0\|_n^2 \leq & 4 \sqrt{\frac{\log n}{n}} \sum_{j \in \cJ_n} \lambda_j |\hat{x}_j-x_{j,0}| + 4\sum_{j \notin \cJ_n} \lambda^2_j |x_{j,0}|^2 \\
\leq & 4 \sqrt{\frac{\log n}{n}} \left(\sum_{j \in \cJ_n} \lambda_j^2\right)^{\frac{1}{2}} \|\tilde{x}_n-x_0\|_n+ 4\sum_{j \notin \cJ_n} \lambda^2_j |x_{j,0}|^\rho |x_{j,0}|^{2-\rho} \\
\leq & 4 \sqrt{\frac{\log n}{n}} \sqrt{|\cJ_n|} \|\tilde{x}_n-x_0\|_n+ 4 \left( \frac{\log n}{n}\right)^{1-\frac{\rho}{2}} \sum_{j \notin \cJ_n} |x_{j,0}|^\rho .
\end{align*}
where $|\cJ_n|$ stands for the cardinal of the set $\cJ_n$. Now using that $$1\geq \sum_j |x_{j,0}|^\rho\geq  \sum_{j \in \cJ_n} |x_{j,0}|^\rho\geq (\f {\log n} n)^{\rho/2} \sum_{j \in \cJ_n} j^{t\rho} \geq (\f {\log n} n)^{\rho/2}|\cJ_n|,  $$
we get the following bound $\sqrt{|\cJ_n|}\leq (\frac{n}{\log n})^{\rho/4}.$
%$$ 4 \sqrt{\frac{\log n}{n}} \|\tilde{x}_n-x_0\|_n\leq 2c_1 n^{-2t} \|\tilde{x}_n-x_0\|_n^2+\frac{2}{c_1} \log(n) \left(\frac{1}{n} \right)^{1-\f \rho 2 -2t},$$
Finally, using that $2xy \leq \gamma x^2+\f 1 \gamma y^2$ for all positive $\gamma$, we obtain that on the set $B_n$ we get
$$ \|\tilde{x}_n-x_0\|_n^2 = O \left[ \log^{1-\f \rho 2}(n) \left(\frac{1}{n}\right)^{1 -\f \rho 2 -2t}
 \right].$$
Assumption $(\Delta)$ ensures that $\bP (B_n^c) \rightarrow 0$, which concludes the proof.
\end{proof}
We point out that the estimator $\tilde{x}_n$ is convergent when the index of ill-posedness is such that $t < \f 1 2 - \f \rho 4$. This implies that for $t> 1/2$, i.e for severe ill-posed problems, the estimation procedure does not lead to consistent estimates. Even in the case where the estimator converges, its rate of convergence is less than the rate of the optimal estimator $\hat{x}_n$. This difference comes from the fact that the standard quadratic loss function is not well fitted to handle inverse problems since the extra term in $\la_j^2=j^{-2t}$ entails a loss of order $n^{-2t}$. That is the reason why we propose a loss function in the definition \eqref{eq:2} which one the one hand gets rid of this term and other hand induces a bias error which can be balanced by the $l^1$ penalty, which enables to achieve optimality. \\ 
\indent We can argue that the regularity assumption \eqref{reg} is more general than Assumption \eqref{Besov2}. However,  increasing the regularity of the function to be estimated does not help increasing the corresponding  rate of convergence of the estimator $\tilde{x}_n$. Indeed the limiting factor in $n^{-2s}$ comes from the rate of decay of the eigenvalues and not the decay of the coefficients $\sum_{j \notin \cJ_n} x_{0,j}^2$.  Hence estimator $\tilde{x}_n$ is outperformed by estimator $\hat{x}_n$.\\
\indent  However Theorem \ref{pasbien} gives the range of application of the standard penalized $l^1$ estimator and provides a better understanding of its behaviour.

\vskip 3mm
\noindent 5. NUMERICAL RESULTS

In this section, we apply our estimation procedure to simulated data obtained using a sequence model. The function we wish to reconstruct is $t \to \sin(1/(.1+t)$, observed in an ill-posed settings with $t$ standing for the index of ill-posedness. We take $n=200$ observations and the observation noise is a Gaussian white noise with ${\rm snr}=2$.\\
 We plot in straight lines the true signal to be recovered, in dotted lines the estimator while the observations are represented by crosses. We will consider two cases depending on the ill-posedness of the operator, with an easy inverse problem $t<\f 1 2$ in Figure \ref{fig1} and a severe ill-posed problem $t>\f 1 2$ in Figure \ref{fig2}. The two figures are significant realizations of 50 replications.
\begin{center}
\begin{figure}
\scalebox{.8}{\includegraphics{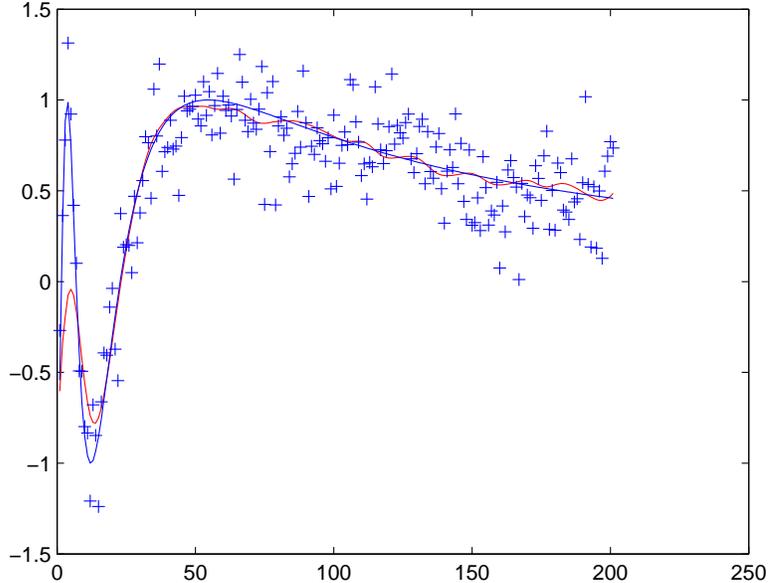}}
\caption{Index of Ill-posedness $t=.2$}
\label{fig1}
\end{figure}
\end{center}
\begin{figure}
\scalebox{.5}{\includegraphics{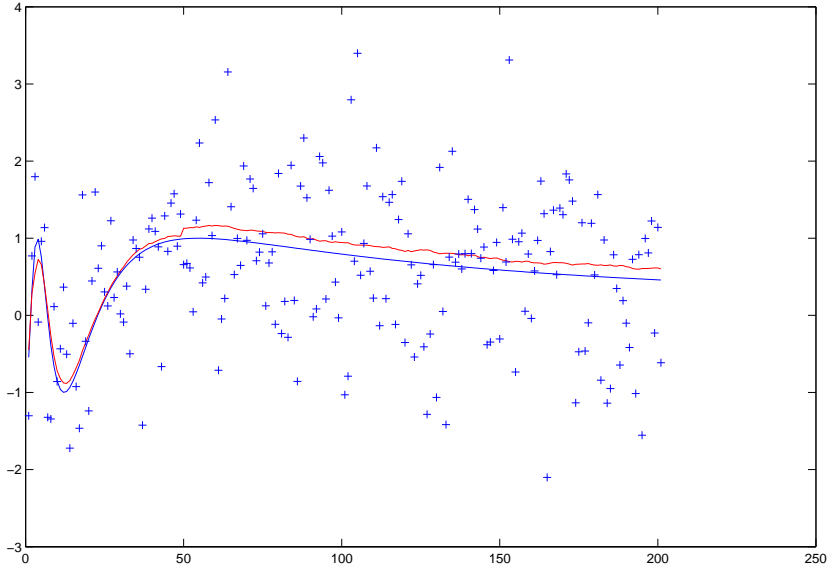}}
\caption{Index of Ill-posedness $t=1.5$}
\label{fig2}
\end{figure}
 We can see that in both cases the estimator $\hat{x}_n$ is consistent and provides a quite good approximation of the unknown function.  Hence the estimation procedure given in \eqref{eq:2} provides a good estimator of ill-posed inverse problems when nothing is known about the regularity of the function. Its main drawback is that a very good knowledge of the operator is needed, but this is the case in most of the denoising procedures for inverse problems. 
\vskip 3mm
\noindent {\sc acknowledgements}
We thank the referees for their carefull reading and valuable comments which contributed to improve this work.
\vskip 3mm
\noindent BIBLIOGRAPHY
\vskip 3mm

\noindent {\scshape O.~V. Besov, V.~P. Iliin, and S.~M. Nikolskii}
  -- \emph{Integral representations of functions and imbedding theorems. {V}ol.
  {I}}, V. H. Winston \&\ Sons, Washington, D.C., 1978, Translated from the
  Russian, Scripta Series in Mathematics, Edited by Mitchell H. Taibleson.

\vskip 3mm

\noindent  {\scshape N.~Bissantz, T.~Hohage and  A.~Munk} -- {\it
  Nonlinear tikhonov regularization for statistical inverse problems},
  \emph{preprint} (2004).

\vskip 3mm

\noindent {\scshape L.~Cavalier, G.~K. Golubev, D.~Picard and A.~B.
  Tsybakov} -- {\it Oracle inequalities for inverse problems }, \emph{Ann.
  Statist.} \textbf{30} (2002), no.~3, p.~843--874, Dedicated to the memory of
  Lucien Le Cam.

\vskip 3mm

\noindent  {\scshape A.~Cohen, M.~Hoffmann and M.~Reiss} -- {\it
  Adaptive wavelet galerkin methods for linear inverse problems },
  \emph{SIAM} \textbf{1} (2003), no.~3, p.~323--354.
\vskip 3mm

\noindent {\scshape S.~Darolles, J.-P. Florens and E.~Renault} --
  {\it Nonparametric instrumental regression }, \emph{preprint} (2003).
\vskip 3mm

\noindent {\scshape I.~Daubechies, M.~Defrise and C.~De~Mol} --
  {\it An iterative thresholding algorithm for linear inverse problems with a
  sparsity constraint }, \emph{Comm. Pure Appl. Math} \textbf{57} (2004),
  p.~1413--1541.
\vskip 3mm

\noindent {\scshape A.~K. Dey, F.~H. Ruymgaart and B.~A. Mair} --
  {\it Cross-validation for parameter selection in inverse estimation
  problems }, \emph{Scand. J. Statist.} \textbf{23} (1996), no.~4,
  p.~609--620.
\vskip 3mm

\noindent {\scshape H.~W. Engl} -- {\it Regularization methods for solving inverse
  problems }, ICIAM 99 (Edinburgh), Oxford Univ. Press, Oxford, 2000,
  p.~47--62.
\vskip 3mm

\noindent {\scshape H.~W. Engl, M.~Hanke and A.~Neubauer} --
  \emph{Regularization of inverse problems}, Mathematics and its Applications,
  vol. 375, Kluwer Academic Publishers Group, Dordrecht, 1996.
\vskip 3mm

\noindent {\scshape A.-K. Fermin, J.-M. loubes and C.~Lude\~na} --
  {\it Model selection for linear inverse problems }, \emph{Proceedings of
  Oberwolfach} \textbf{to appear} (2005).
\vskip 3mm

\noindent {\scshape S.~Van~de Geer} -- \emph{Applications of empirical process theory},
  Cambridge Series in Statistical and Probabilistic Mathematics, Cambridge
  University Press, Cambridge, 2000.
\vskip 3mm

\noindent {\scshape W.~H{\"a}rdle, G.~Kerkyacharian, D.~Picard and
  A.~Tsybakov} -- \emph{Wavelets, approximation, and statistical applications},
  Springer-Verlag, New York, 1998.
\vskip 3mm

\noindent {\scshape B.~Kaltenbacher} -- {\it {Regularization by projection with a
  posteriori discretization level choice for linear and nonlinear ill-posed
  problems.} }, \emph{Inverse Probl.} \textbf{16} (2000), no.~5,
  p.~1523--1539 (English).
\vskip 3mm

\noindent {\scshape J.-M. Loubes and S.~van~de Geer} -- {\it
  Adaptive estimation using thresholding type penalties }, \emph{Statistica
  Neerlandica} \textbf{56} (2002), p.~1--26.
\vskip 3mm

\noindent {\scshape J.-M. loubes and C.~Lude\~na} -- {\it  Penalized
  estimators for nonlinear inverse problems}, \emph{to appear in ESAIM-PS} (2007).
\vskip 3mm

\noindent {\scshape J.-M. Loubes} -- {\it {$\ell\sp 1$} sparsity and applications in
  estimation}, \emph{C. R. Math. Acad. Sci. Paris} \textbf{344} (2007),
  no.~6, p.~399--402.

\noindent {\scshape B.~A. Mair and F.~H. Ruymgaart} -- {\it
  Statistical inverse estimation in {H}ilbert scales }, \emph{SIAM J. Appl.
  Math.} \textbf{56} (1996), no.~5, p.~1424--1444.
\vskip 3mm

\noindent {\scshape U.~Tautenhahn and Q.-n. Jin} -- {\it Tikhonov
  regularization and a posteriori rules for solving nonlinear ill posed
  problems }, \emph{Inverse Problems} \textbf{19} (2003), no.~1, p.~1--21.
\vskip 3mm

\noindent {\scshape A.~N. Tikhonov, A.~S. Leonov and A.~G. Yagola}
  -- \emph{Nonlinear ill-posed problems. {V}ol. 1, 2}, Applied Mathematics and
  Mathematical Computation, vol.~14, Chapman \& Hall, London, 1998, Translated
  from the Russian.

\end{document}